\documentclass[
12pt,reqno]{amsart}
\usepackage{verbatim}
\usepackage{amssymb}

\usepackage{enumerate}
\usepackage[active]{srcltx}
\numberwithin{equation}{section}

\usepackage{t1enc}
\usepackage[utf8x]{inputenc}

\newtheorem{theorem}{Theorem}[section]
\newtheorem{lemma}[theorem]{Lemma}
\newtheorem{corollary}[theorem]{Corollary}

\newtheorem{problem}[theorem]{Problem}

\newtheorem{claim}{Claim}[theorem]

\theoremstyle{definition}
\newtheorem{definition}[theorem]{Definition}

\theoremstyle{remark}

\newcommand{\itemprefix}{}
\makeatletter
\newcommand{\myitem}{%
\item\protected@edef\@currentlabel{\itemprefix\theenumi}%
}
\makeatother

\newcommand{\mc}[1]{\mathcal{#1}}

\newcommand{\setm}{\setminus}
\newcommand{\empt}{\emptyset}
\newcommand{\subs}{\subset}

\renewcommand{\>}{\right\rangle}

\author[A. Dow]{Alan Dow}
\address{Department of Mathematics,
University of North Carolina at Charlotte,
 Charlotte, NC 28223, USA}
\email{adow@uncc.edu}

\author[I. Juh\'asz]{Istv\'an Juh\'asz}
\address      { Alfr\'ed Rényi Institute of Mathematics, Eötvös Loránd Research Network}
\email{juhasz@renyi.hu}

\thanks{The second author was supported by NKFIH grant no. K129211.}



\title{Spaces of countable free set number and PFA}
\date{\today}

\begin{document}

\begin{abstract}
The main result of this paper is that, under PFA, for every {\em regular} space $X$ with $F(X) = \omega$
we have $|X| \le w(X)^\omega$; in particular, $w(X) \le \mathfrak{c}$ implies $|X| \le \mathfrak{c}$.
This complements numerous prior results that yield consistent
examples of even compact Hausdorff spaces $X$ with $F(X) = \omega$ such that $w(X) = \mathfrak{c}$
and $|X| = 2^\mathfrak{c}$.

We also show that regularity cannot be weakened to Hausdorff in this result because we can find in ZFC a Hausdorff
space $X$ with $F(X) = \omega$ such that $w(X) = \mathfrak{c}$
and $|X| = 2^\mathfrak{c}$. In fact, this space $X$ has the {\em strongly anti-Urysohn} (SAU) property that any two
infinite closed sets in $X$ intersect, which is much stronger than $F(X) = \omega$. Moreover, any non-empty
open set in $X$ also has size $2^\mathfrak{c}$, and thus answers one of the main problems of \cite{JShSSz}
by providing in ZFC a SAU space with no isolated points.
\end{abstract}

\maketitle

\section{Introduction}
Following the terminology introduced in \cite{JSSz}, we call a {\em subset} $S \subs X$ {\em free} in $X$ if it admits a well-ordering,
or equivalently an indexing by ordinals, that turns it into a free sequence in $X$. In other words, free sets in $X$ are just the ranges of free sequences in $X$.
Also, we shall use $\mc{F}(X)$ to denote the collection of
all free subsets in $X$. Clearly, then $$F(X) = \sup\{|S| : S \in \mc{F}(X)\},$$
and we call $F(X)$ the {\em free set number} of $X$.

All our other terminology and notation concerning cardinal functions is standard, as e.g. it is in \cite{Ju}.
Our treatment of PFA follows section V.7 of \cite{K1}.

The main result of this paper is that, under PFA, for every {\em regular} space $X$ with $F(X) = \omega$
we have $|X| \le w(X)^\omega$; in particular, $w(X) \le \mathfrak{c}$ implies $|X| \le \mathfrak{c}$.
(By regular we mean regular and Hausdorff.) We will also show that regular cannot be weakened to Hausdorff in this result.

To put our result in perspective, we note that
free sets are obviously discrete, hence we have $F(X) \le s(X)$ for any topological space $X$.
By the classical result of Hajnal and Juhász \cite{HJ}, we have $|X| \le 2^{2^{s(X)}}$ for any Hausdorff space $X$, and
there are many consistent examples showing that this inequality is sharp for $s(X) = \omega$.
For instance, Fedorchuk's celebrated hereditary separable compact space $X$ from \cite{F}, constructed from $\lozenge$,
satisfies $w(X) = \mathfrak{c} = \omega_1$ and $|X| = 2^\mathfrak{c}$.
In \cite{JSh} consistent examples of hereditary separable 0-dimensional spaces $X$ are forced, with both $\mathfrak{c}$ and $2^\mathfrak{c}$
as large as you wish, independently of each other, such that $w(X) = \mathfrak{c}$ and $|X| = 2^\mathfrak{c}$.

On the other hand, Todorcevic proved in \cite[Theorem 11]{T}  that PFA implies $|X| \le \mathfrak{c}$ for any Hausdorff space $X$ with $s(X) = \omega$.
While this fails if we only have $F(X) = \omega$, even if in addition $w(X) = \mathfrak{c}$ holds, we do get
$|X| \le \mathfrak{c}$ from PFA for regular $X$ with $F(X) = \omega$ and $w(X) \le \mathfrak{c}$.

\section{A ZFC result}

There seems to be basically only one ZFC method of constructing free sequences (sets) that is the main lemma 2.1
of \cite{JSSz}. We repeat it here because we shall use it several times.

\begin{lemma}\label{lm:freeinA}
Assume that $X$ is a space, $A \subs X$, $\,{\kappa}$ is an infinite cardinal, and
$\mc W\subs \tau(X)$,  moreover
\begin{enumerate}[(a)]
\item $\mc W$ is closed under unions of subfamilies of size $<\,\kappa$,
\item $A\setm W\ne \empt$ for each $W\in \mc W$,
\item for each $S \subs A$ with $S\in \mc{F}(X)$ and $|S| < \kappa$ there is $W\in \mc W$
with $\overline{S}\subs W$.
\end{enumerate}
Then there is a subset of $A$ of size $\kappa$ that is free in $X$.
\end{lemma}

Our next result is an easy consequence of Lemma \ref{lm:freeinA}. We recall that a $G_\kappa$-set
is one obtainable as the intersection of at most $\kappa$ open sets.

\begin{lemma}\label{lm:Gdelta}
Assume that $X$ is any space, $z \in X$, and $\kappa$ is an infinite cardinal, moreover for every
open $U$ with $z \in U$ there is a {\em closed} $G_\kappa$-set $H$ such that $z \in H \subs U$.
(Clearly, this holds true if $X$ is regular.)
If, in addition, $Y \subs X \setm \{z\}$ is such that
\begin{enumerate}[(i)]
\item $H \cap Y \ne \emptyset$ for every $G_\kappa$-set $H$ with $z \in H$,
\item $z \notin \overline{S}$ for every $S \in \mc{F}(X) \cap [Y]^{\le \kappa}$,
\end{enumerate}
then $\mc{F}(X) \cap [Y]^{\kappa^+} \ne \emptyset$, i.e. there is a subset of $Y$ of size $\kappa^+$ that is free in $X$.
\end{lemma}

\begin{proof}
We may apply Lemma \ref{lm:freeinA} with $Y$ instead of $A$, $\kappa^+$ instead of $\kappa$, and
with $\mathcal{W}$ consisting of all open $F_\kappa$-sets $U$ such that $z \notin U$.
\end{proof}

Before presenting the main result of this section, we need to introduce the following piece of notation.

\begin{definition}
For any (infinite) cardinal $\kappa$,
$${\bf{h}}(\kappa) = \sup\{|X| : X \text{ is regular with $F(X) = \omega$ and }d(X) \le \kappa\}.$$
\end{definition}

Note that we trivially have $2^\kappa \le {\bf{h}}(\kappa) \le 2^{2^{\kappa}}$.

\begin{theorem}\label{tm:h}
For every regular space $X$ with $F(X) \le \kappa$ we have $$|X| \le \big(w(X) \cdot {\bf{h}}(\kappa)\big)^\kappa.$$
\end{theorem}

\begin{proof}
Let us put $\mu = w(X) \cdot {\bf{h}}(\kappa)$ and then fix an open base $\mathcal{B}$ of $X$ with $|\mathcal{B}| = w(X) \le \mu$.
Next we consider an elementary submodel
$M$ of $H(\vartheta)$ for a large enough regular cardinal $\vartheta$ such that $|M| = \mu^{\kappa}$, $\,M$ is
${\kappa}$-closed, and $X, \mathcal{B} \in M$. We shall show that $X \subs M$.

Indeed, assume on the contrary that $X \cap M = Y$ and $z \in X \setm Y$. Now, if $H$ is any $G_\kappa$-set with
$z \in H$ then there is some $\mathcal{C} \in [\mathcal{B}]^{\le \kappa}$ such that $z \in C = \cap \mathcal{C} \subs H$.
But $\mathcal{B} \subs M$ and the $\kappa$-closedness of $M$ imply $\mathcal{C} \in M$, and so $C \cap M = C \cap Y \ne \emptyset$
by elementarity and $z \in C$, hence $H \cap Y \ne \emptyset$ as well.

Next, for every subset $S \in [Y]^{\le \kappa}$ we have $|\overline{S}| \le {\bf{h}}(\kappa) \le \mu$ by definition, moreover
$S \in M$ and hence $\overline{S} \in M$ as well. But then we also have $\overline{S} \subs Y$. This means that
both conditions of Lemma \ref{lm:Gdelta} are satisfied, hence there is a subset of $Y$ of size $\kappa^+$ that is free in $X$.
But this contradicts $F(X) \le \kappa$, completing our proof.
\end{proof}

\section{Some consequences of PFA}

We start this section with a general theorem that gives conditions which,
under PFA, imply the existence of an uncountable free set.

\begin{theorem}\label{tm:PFA}
Assume PFA. Let $X$ be a topological space, $Y \subs X$ its subspace, and $\mathcal{A} \subs [Y]^\omega$ satisfying the following three conditions.
\begin{enumerate}[(1)]
\item For every countable $\mathcal{A}_0 \subs \mathcal{A}$ we have $$Z(\mathcal{A}_0) = \bigcap \{\overline{A} : A \in \mathcal{A}_0\} \cap Y \ne \emptyset.$$
\item For every $y \in Y$ there are an open $U_y$ with $y \in U_y$ and $A_y \in \mathcal{A}$ such that $\overline{U_y} \cap \overline{A_y} = \emptyset.$
\item Let $\mathcal{H}$ be the collection of all $H \subs Y$ intersecting $Z(\mathcal{A}_0)$ for all $\mathcal{A}_0 \in [\mathcal{A}]^{\le \omega}$.
For every $H \in \mathcal{H}$ there is $A \in \mathcal{A}$ with $A \subs H$.
\end{enumerate}
Then $Y$ has an uncountable subset that is free in $X$.
(All closures above are taken in $X$.)
\end{theorem}

\begin{proof}
We start by fixing a well-ordering $\prec$ of $Y$ and then a large enough regular cardinal $\kappa$ such that
$H(\kappa)$ contains all the objects above.
We shall say that $M$ is $suitable$
if it is a countable elementary submodel of $H(\kappa)$ and contains $\big\{X, Y, \prec, \mathcal{A}, \{\<U_y, A_y\> : y \in Y\}, \mathcal{H}\big\}$.
We shall denote by $\mathcal{S}$ the collection of all suitable $M$'s.

For every $M \in \mathcal{S}$ let $y(M)$ be the $\prec$-minimal member of $Z(M \cap \mathcal{A}).$
Note that if $M, N \in \mathcal{S}$ with $M \in N$ then $Z(M \cap \mathcal{A}) \supset Z(N \cap \mathcal{A})$
and $y(M) \in N$ imply $y(M) \prec y(N)$.

Now we are ready to define the partial order $\mathbb{P} = \<P, <\>$ that will be used to prove our result.
The elements of $P$ will be all finite $\epsilon$-chains of members of $\mathcal{S}$. Clearly, for every
$p \in P$, if $p \ne \emptyset$ then $\cap p$ is the bottom and $\cup p$ is the top member of $p$.

To define $<\,$, we first introduce the following notation. If $N \in p \in P$ then we let
$$W(p, N) = \bigcap \{U_{y(M)} : M \in p \text{ and } y(N) \in U_{y(M)} \}.$$
Clearly, then $y(N) \in W(p, N) \subs U_{y(N)}$ and $N \in M \in p$ implies $y(M) \notin U_{y(N)}$ by condition (2).

Now, by definition, $p < q$ holds for for $p,q \in P$ iff $p \supset q$ and for every $M \in p \setm q$
with $M \in \cup q$ we have $y(M) \in W(q,N)$, where $N$ is the minimal element of $q$ such that $M \in N$.

We have to check that $<\,$ is transitive. So, assume that $r < p < q$, $M \in r \setm q$ with $M \in \cup q$,
moreover $N$ is the minimal element of $q$ such that $M \in N$. Now we distinguish two cases.

First, if $N$ is also the minimal element of $p$ such that $M \in N$ then we have $y(M) \in W(p,N) \subs W(q,N)$,
using that $q \subs p$. Otherwise, the minimal $K \in p$ containing $M$ satisfies $M \in K \in N$,
while $N$ is also the minimal element of $q$ with $K \in N$. So, by definition, $p < q$ implies $y(K) \in W(q,N)$
that clearly implies $W(p,K) \subs W(q,N)$. But by $r < p$ we have $y(M) \in W(p,K)$, hence $y(M) \in W(q,N)$ as well.

To be able to apply PFA, we also have to show that $\mathbb{P}$ is proper. To do that, we choose a large enough regular
cardinal $\vartheta$ such that $\mathbb{P} \in H(\vartheta)$. Clearly, $\vartheta > 2^\kappa$ will do.
We intend to show that for every countable elementary submodel $\mathcal{M}$ of $H(\vartheta)$ with $\mathbb{P} \in \mathcal{M}$
and for every condition $p_0 \in P \cap \mathcal{M}$ there is a condition $q_0 < p_0$
which is $(\mathcal{M},\mathbb{P})$-generic.

It is standard to check that $N = \mathcal{M} \cap H(\kappa) \in \mathcal{S}$, hence $\{N\} \in P$. More generally,
for any $p_0 \in P \cap \mathcal{M}$ we have $p_0 \in N$, hence $q_0 = p_0 \cup \{N\} \in P$ and $q_0 < p_0$.
We claim that this $q_0$ is $(\mathcal{M},\mathbb{P})$-generic, i.e.
for every dense set $D$ in $\mathbb{P}$, if $D \in \mathcal{M}$ then $D \cap \mathcal{M}$ is predense below $q_0$.
Now, this is trivially implied by the following key lemma.

\begin{lemma}\label{lm:gen}
For every dense set $D$ in $\mathbb{P}$ with $D \in \mathcal{M}$,
if $N \in p \in D$, where $N=\mathcal{M}\cap H(\kappa)$, then there is $r \in D \cap \mathcal{M}$
such that $r \cup p < p$.
\end{lemma}

\begin{proof}
To start with, we fix the dense $D \in \mathcal{M}$ and $p \in D$ with $N \in p$.
Then we write $p^- = p \cap \mathcal{M}$ and
$p \setm \mathcal{M} = \{N = N_0\, \epsilon N_1\, \epsilon ...\, \epsilon N_k\}$.
It will be convenient to put $P_k = \{p \in P : |p| = k+1\}$.
Clearly, $p^- \in P \cap \mathcal{M}$ and $p \setm \mathcal{M} \in P_k$.
Also, each $q \in P_k$ is of the form $q = \{M_{q,0}\, \epsilon M_{q,1}\, \epsilon ...\, \epsilon M_{q,k}\}$.

The $r \in D \cap \mathcal{M}$ that we need will be of the form $r = p^-\, \cup q$ for some $q \in P_k \cap \mathcal{M}$
with $p^-\, \in M_{q,0}$. Then, to have $r \cup p < p$, what we need is that $y(M_{q,i}) \in W(p,N)$ for all $i \le k$.
To handle this, we shall write $$s_q = \<y(M_{q,i}) : i \le k\>$$ for any $q \in P_k$.

Since $p^-, D, P_k \in \mathcal{M}$,
so is $$E_0 = \{q \in P_k :  p^- \in \cap q = M_{q,0} \text{ and } p^- \cup q \in D\},$$ as well as
$L_0 = \{s_q : q \in E_0\}$. We also have $L_0 \in H(\kappa)$, hence $L_0 \in N = \mathcal{M} \cap H(\kappa)$.
We clearly have $p \in E_0$, hence $s_p \in L_0$. Finally, let $T_0$ be tree consisting of all initial
segments of members of $L_0$, formally $$T_0 = \{t \upharpoonright i : t \in L_0 \text{ and } i \le k + 1\}.$$
Thus $L_0$ is the top level of $T_0$; our trees grow upwards.

Next we are going to recursively prune $T_0$ in $k$ steps to obtain the trees $T_0 \supset T_1 \supset ... \supset T_k$
in such a way that $s_p \in T_i \in \mathcal{M}$ and hence $T_i \in N$ for all $i \le k$.
Of course, we already know these for $i= 0$.

To prepare this recursive pruning, we introduce some new notation. First, we are going to denote by $\widehat{\mathcal{A}}$
the family of all those subsets of $Y$ that include some member of $\mathcal{A}$. Note that condition (3) of our theorem
simply says that $\mathcal{H} \subs \widehat{\mathcal{A}}$.
Once we have the tree $T_i$ and $t \in T_i$ we are going to write $$[t]_i = \{s \in T_i : t \subs s \text{ or } s \subs t\},$$
moreover $suc_i(t)$ will denote the set of immediate successors of $t$ in $T_i$.

We shall also need the following simple claim.

\begin{claim}
For every suitable $M \in \mathcal{S}$ and $H \subs Y$, if $y(M) \in H \in M$ then $H \in \mathcal{H}$.
\end{claim}

Indeed, if we had $H \notin \mathcal{H}$ then, by elementarity, there would be some $\mathcal{A}_0 \in [\mathcal{A}]^{\le \omega} \cap M$
with $H \cap Z(\mathcal{A}_0) =  \emptyset$.
But then we would have $y(M) \in Z(M \cap \mathcal{A}) \subs Z(\mathcal{A}_0)$, a contradiction.

Now, to get $T_1$ from $T_0$, we first define $$L_1 = \big\{t \in T_0 : |t| = k \text{ and } suc_0(t) \in \widehat{\mathcal{A}} \big\},$$
and then put $T_1 = \cup \{ [t]_0 : t \in L_1\}$. It is clear that then $T_0 \in N$ implies $L_1, T_1 \in N$.
To see that $s_p \in T_1$, we have to show that $t_p = s_p \upharpoonright k \in L_1$. But this follows from
$y(M_{p,k}) \in suc_0(t_p) \in M_{p,k}$ and the above Claim because $suc_0(t_p) \in \mathcal{H} \subs \widehat{\mathcal{A}}$.

The general recursive step from $T_i$ to $T_{i+1}$ (for $i < k$) is very similar.  Given $T_i \in N$, we first define
$$L_{i + 1} = \big\{t \in T_i : |t| = k + 1 - i \text{ and } suc_i(t) \in \widehat{\mathcal{A}} \big\},$$
and then put $T_{i +1} = \cup \{ [t]_i : t \in L_{i + 1}\} \subs T_i$. By induction, then $T_i \in N$ implies $L_{i+1}, T_{i+1} \in N$,
moreover if $t_p = s_p \upharpoonright k - i$ then $y(M_{p,k - i}) \in suc_i(t_p) \in M_{p,k-i}$ and the Claim imply $t_p \in L_{i + 1}$,
and hence $s_p \in T_{i+1}$.

So, after having completed all the $k$ steps, we arrive at the tree $T_k$ which clearly has the following property:
For every $t \in T_k$ if $|t| \le k$ then $suc_k(t) \in \widehat{\mathcal{A}}$.

Now we are going to show that $T_k$ has a member $s = \<y_i : i \le k\> \in N$ such that $y_i \in W(p, N)$ for all $i \le k$.
First, to find $y_0$, we use $\emptyset \in N \cap T_k$  and $suc_k(\emptyset) \in N \cap \widehat{\mathcal{A}}$ to obtain
$A_0 \in N \cap \mathcal{A}$ such that $A_0 \subs suc_k(\emptyset)$. We then have $y(N) \in \overline{A_0}$, hence $W(p, N) \cap A_0 \ne \emptyset$.
But, as $A_0$ is countable, we also have $A_0 \subs N$, hence any $y_0 \in W(p, N) \cap A_0$ is in $N \cap suc_k(\emptyset)$.
We may go on like this by induction. Given $s_j = \<y_i : i < j\> \in N \cap T_k$ such that $y_i \in W(p, N)$ for all $i < j$ for some $0 < j \le k$,
we use $suc_k(s_j) \in N \cap \widehat{\mathcal{A}}$ to obtain $A_j \in N \cap \mathcal{A}$ such that $A_j \subs N \cap suc_k(s_j)$.
Then $y(N) \in \overline{A_j}$, and hence $W(p, N) \cap A_j \ne \emptyset$ yields us $y_j \in N \cap W(p, N) \cap suc_k(s_j)$.

We have $T_k \subs T_0$, so $s \in T_0 \cap \mathcal{M}$, and this means that by elementarity there is some $q \in P_k \cap \mathcal{M}$
such that $s = s_q$ and $r = p^- \cup q \in D \cap \mathcal{M}$. But the choice of $s = s_q$ then allows us to conclude that $r \cup p < p$,
completing the proof of the Lemma.
\end{proof}

\medskip

Now we turn to the much easier task of finding $\omega_1$ dense sets in $\mathbb{P}$ such that a filter in $\mathbb{P}$ meeting all of them
gives us an uncountable subset of $Y$ that is free in $X$. First we show that for every countable ordinal $\alpha$ the set
$D(\alpha) = \{p \in P : \alpha \in \cup p\}$ is dense in $\mathbb{P}$. Indeed, for any $p \in P$ clearly there is $M \in \mathcal{S}$
such that $\{p, \alpha\} \subs M$. But then $p > q = p \cup \{M\} \in D(\alpha)$.

So, assume that $G \subs P$ is a filter in $\mathbb{P}$ such that $G \cap D(\alpha) \ne \emptyset$ for all $\alpha \in \omega_1$.
Clearly, then $\cup G \subs \mathcal{S}$ is an uncountable $\epsilon$-chain that is well-ordered by $\epsilon$ in type $\omega_1$.
Since $\{M, N\} \subs \mathcal{S}$ with $M \in N$ implies $y(M) \prec y(N)$, it follows that $S = \{y(M) : M \in \cup G\}$
is also well-ordered by $\prec$ in type $\omega_1$.

Since $G$ is a filter in $\mathbb{P}$, if $\{M, N\} \subs G$ with $M \in N$ then we have $\{M, N\} < \{N\}$,
hence $y(M) \in U_{y(N)}$ by the definition of $<\,$.
On the other hand, if $N \in M$ then $y(M) \in \overline{A_{y(N)}}$. Thus $\overline{U_{y(N)}} \cap \overline{A_{y(N)}} = \emptyset$
implies that $$\overline{\{y(M) : M \in \cup G \cap N\}} \cap \overline{\{y(M) : M \in \cup G \text{ and }N \in M \}} = \emptyset$$
for every $N \in \cup G$. So, the set $S^+$ of all successor members of $S$ under $\prec$ is free in $X$.
\end{proof}

We are now ready to present our promised main result.

\begin{theorem}\label{tm:main}
Under PFA we have $|X| \le w(X)^\omega$ for every regular space $X$ with $F(X) = \omega$.
\end{theorem}

\begin{proof}
We are going to prove the contrapositive of the statement: If $X$ is regular with $|X| > w(X)^\omega$ then $F(X) > \omega$.
Let us fix an open base $\mathcal{B}$ of $X$ with $|\mathcal{B}| = w(X)$ and then
consider an elementary submodel
$M$ of $H(\vartheta)$ for a large enough regular cardinal $\vartheta$ such that $|M| = w(X)^\omega$, $\,M$ is
$\omega$-closed, and $X, \mathcal{B} \in M$. Let $Y = M \cap X$ and pick $z \in X \setm Y$.
Since $\mathcal{B} \in M$ and $M$ is $\omega$-closed, we clearly have $z \in \overline{Y}^\delta$, i.e. $G \cap Y \ne \emptyset$ for any
$G_\delta$-set $G$ with $z \in G$.

Now we distinguish two cases. First, if there is $H \subs Y$ such that  $z \in \overline{H}^\delta$ but $z \notin \overline{S}$
for all $S \in [H]^\omega$ then $H$ has an uncountable subset free in $X$ by Lemma \ref{lm:Gdelta}, hence we are done.
(This part does not use PFA.)

So, we may assume that, putting $\mathcal{A} = \{A \in [Y]^\omega : z \in \overline{A}\}$, every
$H \subs Y$ with  $z \in \overline{H}^\delta$ includes a member of $\mathcal{A}$. In particular, we have $\mathcal{A} \ne \emptyset$.

Now, if $\mathcal{A}_0$ is any countable subfamily of $\mathcal{A}$ then by the $\omega$-closure of $M$ and by
elementarity we have $Z(\mathcal{A}_0) = \bigcap \{\overline{A} : A \in \mathcal{A}_0\} \cap Y \ne \emptyset.$
So, condition (1) of Theorem \ref{tm:PFA} is satisfied.

Fix $A \in \mathcal{A}$ and for every $y \in Y$ pick open $U_y$ with $y \in U_y$ and $V_y$ with $z \in V_y$ such that
$\overline{U_y} \cap \overline{V_y} = \emptyset$. Then for $A_y = A \cap V_y \in \mathcal{A}$ we have
$\overline{U_y} \cap \overline{A_y} = \emptyset$. This means that condition (2) of Theorem \ref{tm:PFA} is also satisfied.

Finally, assume that $H \in \mathcal{H}$, i.e. $H \cap Z(\mathcal{A}_0) \ne \emptyset$ for all countable $\mathcal{A}_0 \subs \mathcal{A}$.
We claim that then $z \in \overline{H}^\delta$.
By the regularity of $X$, it suffices for this to show that $z \in G$ implies $G \cap H \ne \emptyset$ for any $G_\delta$-set
of the form $G = \cap_{n < \omega}U_n$, where  $\overline{U_{n+1}} \subs U_n$ for all $n < \omega$.
But for any fixed $A \in \mathcal{A}$ then $\mathcal{A}_0 = \{A \cap U_n : n < \omega\} \subs \mathcal{A}$,
moreover $Z(\mathcal{A}_0) \subs G$, hence $G \cap H \ne \emptyset$.
But by our assumption then $H$ includes a member of $\mathcal{A}$, hence condition (3) of Theorem \ref{tm:PFA} is also satisfied.
Consequently, we get an uncountable subset of $Y$ that is free in $X$ by Theorem \ref{tm:PFA}, in this case as well.
\end{proof}

Since $w(X) \le 2^{d(X)}$ for any regular space $X$ and PFA implies $\mathfrak{c} = 2^{\omega_1} = \omega_2$ , we have
${\bf{h}}(\omega) = {\bf{h}}(\omega_1) =\omega_2$ under PFA, as an immediate corollary of Theorem \ref{tm:main}.
These lead to the following results, whose easy proofs are left to the reader.

\begin{corollary}\label{co:F=t}
Assume PFA and let $X$ be a regular space.
\begin{enumerate}[(i)]
  \item If $F(X) = \omega$ and $t(X) \le \omega_1$ then $d(X) \le \mathfrak{c}$ implies $|X| \le \mathfrak{c}$.
  \item If $F(X) = t(X) = \omega$ then $|X| \le d(X)^\omega$.
\end{enumerate}
\end{corollary}

Since $F(X) = t(X)$ for compact $X$, part (ii) implies  $|X| \le d(X)^\omega$ for countably tight compact $X$.
This, of course was known as a consequence of Balogh's classical result that countably tight compacta are sequential
under PFA, see \cite{B}.

We close this section by raising two related questions that we could not answer.

\begin{problem}
Does PFA imply $|X| \le d(X)^\omega$ for each regular space $X$ with $F(X) = \omega$?
Or, at least, does PFA imply $|X| \le \mathfrak{c}$ if $d(X) \le \mathfrak{c}$
\end{problem}

\begin{problem}
Can Theorem \ref{tm:main} be extended from regular to Urysohn spaces?
\end{problem}

\section{The Hausdorff case}

In this section we are going to present a ZFC example of a Hausdorff space $X$ such that $F(X) = \omega$, $\,w(X) \le \mathfrak{c}$,
and $|X| = 2^{\mathfrak{c}}$. So, this will show that regularity cannot be weakened to Hausdorff in our main result \ref{tm:main}.
The example is non-trivial, however, fortunately for us, we only need to perform a minor modification of the example from \cite{JShSSz}
where the hard work was done.

Now, the example from \cite{JShSSz} is a separable {\em strongly anti-Urysohn} (SAU) space $X$ of cardinality $2^{\mathfrak{c}}$.
The SAU property means that any two infinite closed sets in $X$ intersect. This then trivially implies $F(X) = \omega$ because
then every free sequence in $X$ has order type less than
$\omega + \omega$. (Actually, a SAU space must also have at least two non-isolated
points but as this $X$ is separable, it has only countably many isolated points.)

The reason why the space $X$ from \cite{JShSSz} needs to be modified for our purposes is that its weight is $2^{\mathfrak{c}}$.
So, the space we need will be $X_r$, whose topology $\tau_r$ is the coarser topology on $X$ generated by $RO(X)$, the family
of all regular open sets in $X$. Then $X_r$ is Hausdorff because $U \cap V = \emptyset$ for any open $U, V$ in $X$ implies
$Int \overline{U} = Int \overline{V} = \emptyset$. Also, $X_r$ is SAU because $X$ is, hence $F(X_r) = \omega$. Finally, since
$X$ is separable, we have $|RO(X)| = \mathfrak{c}$, consequently $\,w(X) \le \mathfrak{c}$.

The space $X$ from \cite{JShSSz} is right-separated, i.e. scattered and, although many consistent examples of crowded SAU spaces had
been constructed, their existence in ZFC was not known. So, it was asked explicitly in \cite{JShSSz} if
they exist. Now, $X_r$ has the same isolated points as $X$, so it is not crowded but using $\,w(X) \le \mathfrak{c}$
we can actually get such an example, thus giving an affirmative answer to this problem from \cite{JShSSz}.

\begin{theorem}
$X_r$ has a closed, hence also SAU, subspace $Y$ such that $\Delta(Y) = 2^{\mathfrak{c}}$, i.e. every non-empty open set in $Y$
has cardinality $2^{\mathfrak{c}}$. In particular, $Y$ is crowded.
\end{theorem}

\begin{proof}
Let $\mathcal{U}$ be the family of all those open sets in $X_r$ that have size $< 2^{\mathfrak{c}}$. Since $hL(X_r) \le w(X_r) \le \mathfrak{c}$,
there is a subfamily $\mathcal{V} \subs \mathcal{U}$ with $|\mathcal{V}| \le \mathfrak{c}$ such that $W = \cup\, \mathcal{U} = \cup\,\mathcal{V}$.
But we have $cf(2^{\mathfrak{c}}) > \mathfrak{c}$, hence $|W| < 2^{\mathfrak{c}}$. Now, it is obvious that $Y = X_r \setm W$ is as required.
\end{proof}

Thus only just one open problem is left concerning SAU spaces that we cannot resist to repeat here: Is it provable in ZFC
that every SAU space has cardinality at most $2^{\mathfrak{c}}$? It was proved in \cite{JSSz} that $2^{2^{\mathfrak{c}}}$
is an upper bound.

\end{document}